\documentclass{aptpub}

\authornames{LI, VU, ZHOU} 
\shorttitle{Exit problems for general draw-down times} 

\usepackage{amsmath,amssymb,mathrsfs}
\usepackage{mathrsfs,amsfonts}
\usepackage[colorlinks,anchorcolor=blue,citecolor=blue]{hyperref}
\usepackage{color}

\def\bE{\mathbb{E}}
\def\bP{\mathbb{P}}
\def\non{\nonumber}
\def\D{\displaystyle}

\begin{document}

\title{Exit problems for general draw-down times of spectrally negative L{\'e}vy processes} 

\authorone[Nankai University]{Bo Li} 
\addressone{School of Mathematics and LPMC, Nankai University, China.} 
\authortwo[Concordia University]{Nhat Linh Vu}
\addresstwo{Department of Mathematics and Statistics, Concordia University, Canada.}
\authorthree[Concordia University]{Xiaowen Zhou}
\addressthree{Department of Mathematics and Statistics, Concordia University, Canada.}


\begin{abstract}
For spectrally negative L\'evy processes, we prove several fluctuation results involving a general draw-down time,
which is a downward exit time from a dynamic level that depends on the running maximum of the process.
In particular, we find expressions of the Laplace transforms for the two-sided exit problems
involving the draw-down time. We also find the Laplace transforms for the hitting time and the creeping time over the running-maximum related
draw-down level, respectively,   and obtain an expression for a draw-down  associated potential measure. The results are expressed in terms of scale functions for the spectrally negative L\'evy processes.
\end{abstract}



\section{Introduction}

As part of fluctuation theory, exit problems for spectrally negative
L\'evy processes and the associated reflected processes have been
studied extensively over the past ten years. Such
problems often concern the joint distributions of the process when it
first leaves either a finite or a semi-finite interval, or when its
draw-down (from the running maximum) or its draw-up (from the
running minimum) first exceeds a fixed level. These results are
often expressed in terms of the scale functions. We refer to
\cite{Kyprianou2014:book:levy} and references therein for a
collection of such results; also see
\cite{Pistorius2004:passagetime:reflect},
\cite{Avram2004:levyexit:annals}, \cite{Zhou08} and \cite{LLZ15}
for research on the draw-down times and the draw-up times for spectrally
negative L\'evy processes reflected at the running
maximum and the running minimum processes, respectively. We refer to \cite{Zhang14}
for recent work on draw-down(up) times of regular diffusions.

Exit problems involving more general first passage times had been
considered earlier for time homogeneous diffusions. In \cite{Leh77}
an exit problem with exit level depending on the running maximum of
the diffusion was studied and a joint Laplace transform was found for
such a general draw-down time. The general draw-down times find
interesting applications in \cite{AY79} of defining the Az\'ema-Yor
martingales to solve the Skorokhod embedding problem.
The draw-down problem for renewal processes was studied in \cite{Bin2017}.
More recent
work on applications of draw-down times can be found in \cite{CKO12,Bin2017}
and references therein.

The general draw-down times for spectrally negative L\'evy processes
had been studied in \cite{Pis07} with an excursion theory approach
to obtain the Skorohod embedding for spectrally negative L\'evy
process and for the associated reflected process from its maximum.
Following a similar approach,
\cite{MIJATOVIC20123812} further considered a sextuple law related to the draw-down time.
In \cite{ALZ17} we considered a perturbed spectrally negative L\'evy risk process, the so called L\'evy tax process, with a
draw-down exit level that is a linear function of the running maximum,
and found expressions on the present values of amount of tax
for this process. To this end,
we applied both the excursion theory and an approximation approach using solutions to the exit
problems with fixed boundaries.

In this paper we continue to investigate the exit problems for a spectrally negative L\'evy process with
a dynamic draw-down exit level that depends on the running maximum in a general way.
Applying the excursion theory  which comes in handy in analyzing the the draw-down fluctuation behaviors for the spectrally negative L\'evy processes,  we first find the expressions for a joint Laplace transform of the process at the draw-down time. We also find the Laplace transforms for the hitting time and the creeping time of the draw-down level, respectively. In addition, we obtain an expression for a potential measure associated to the draw-down time.

This paper is structured as follows. After the introduction, in Section 2 we review basic facts on spectrally negative L\'evy processes to prepare for the main proofs. The main results are presented in Section 3. Section 4 focuses on results with linear draw-down functions where the results are more apparent and the previously known results can be easily recovered. Proofs of the main results are deferred to Section 5.

\section{Spectrally negative L\'evy processes}\label{sec:2}

Throughout this paper, let $X=\{X_{t}, t\geq 0\}$ be a spectrally negative
L\'{e}vy process (SNLP for short), i.e. a stochastic process with stationary
independent increments and with no positive jumps, defined on a
filtered probability space $(\Omega, \mathscr{F}, (\mathscr{F}_{t})_{t\geq 0}
,\mathbb{P})$. We also assume that $X$ is not the negative of
a subordinator. Denote by $\mathbb{P}_{x}$ the probability law of
$X$ given $X_{0}=x$, and the corresponding expectation by
$\mathbb{E}_{x}$. Write $\mathbb{P}$ and $\mathbb{E}$ when $x = 0$.
Its Laplace transform always exists with the Laplace exponent given
by
 $$ \psi (\lambda):= \frac{1}{t}\log\mathbb{E}\big(e^{\lambda X_{t}}\big)\quad \text{for}\quad \lambda\geq 0$$
 where
 $$\psi (\lambda)=\mu \lambda +\frac{1}{2} \sigma^{2} \lambda^{2} + \int_{(-\infty,0)}(e^{\lambda x}-1-\lambda x\mathbf{1}_{\{x>-1\}})\Pi ( dx)$$
for $\mu\in\mathbb{R},\sigma \geq 0$ and the $\sigma$-finite L\'{e}vy measure $\Pi$ on $(-\infty,0)$ satisfying $\int_{(-\infty,0)}(1\wedge x^{2})\Pi( dx) < \infty$. Further, there exists a function
$\Phi: [0,\infty)\to [0,\infty)$
defined by
$$
 \Phi (q) :=\sup\{\lambda \geq0 : \psi (\lambda)=q\}\quad\quad\quad \text{for}\quad q \geq 0.
$$

Scale functions play a central role in the fluctuation theory for SNLPs.
For $q\geq0$, the $q$-scale function $W^{(q)}$ of $X$ is defined as
a function continuous and increasing on $[0,\infty)$ and satisfying
\begin{eqnarray}\label{1}
\int_{0}^{\infty}e^{-\lambda y}W^{(q)}(y) dy=\frac{1}{\psi (\lambda)-q}
\quad\text{for $\lambda>\Phi(q)$}.
\end{eqnarray}
For convenience, we extend the domain of $W^{(q)}$ to the whole real line by setting
$W^{(q)}(x)=0$ for all $x<0$.
Given $W^{(q)}$, the second scale function is defined by
$$Z^{(q)}(x):=1+q\int_0^x W^{(q)}(y) dy.$$
Write $ W= W^{(0)}$ and $Z=Z^{(0)}$ whenever $q=0$.

It is well-known that $W^{(q)}(x)$ on $\mathbb{R}^{+}$ is continuous and
strictly increasing. $W^{(q)}(0)=W(0)>0$ if and only if the process $X$
has paths of bounded variation, and if and only if $\sigma=0$ and
$\int_{-1}^{0} |x|\Pi(dx)<\infty$.
The scale function $W^{(q)}$ is continuously differentiable on $(0,\infty)$
if the process $X$ has paths of unbounded variation
(and in particular, if it has a nontrivial Gaussian component)
or if the process $X$ has paths of bounded variation
and the L{\'e}vy measure has no atoms. Moreover, if
$\sigma>0$, $W^{(q)}$ has continuous derivative of order two on
$(0,\infty)$ and $W^{\prime}(0+)={2}/{\sigma^{2}}$. We refer the
readers to \cite{Chan2011:scalefunction:smooth} for more detailed
discussions on the smoothness of scale functions.

For $c\geq0$, process $\{e^{cX_{t}-\psi(c)t},  t\geq0\}$ is a martingale under $\bP$.
Introduce a new probability measure satisfying
$$
\frac{d\bP^{(c)}}{d\bP}\Big|_{\mathscr{F}_{t}}
= e^{c X_{t}-\psi(c)t}\quad\text{for every $t\geq0$}.
$$
It is well known that $X$ is still a SNLP under $\bP^{(c)}$.
Denoting the associated Laplace exponent and scale functions with a subscript $c$ under $\bP^{(c)}$, a straightforward calculation shows that, for $c\geq0, q+\psi(c)\geq0$,
$$
\psi_{c}(s)=\psi(c+s)-\psi(c)
\quad\text{and}\quad
\Phi_{c}(s)=\Phi(s+\psi(c))-c
\quad \text{for}\quad s\geq0;
$$
in addition,
$$
W^{(q)}_{c}(x)=e^{-cx} W^{(q+\psi(c))}(x)
\quad\text{and}\quad
Z^{(q)}_{c}(x)=1+ q \int_{0}^{x}W^{(q)}_{c}(y) dy.
$$
Notice that, for $x\geq0$, $W^{(q)}(x)$ and $Z^{(q)}(x)$ are analytically extendable to all $q\in\mathbb{C}$, \cite[Lem 8.3, Cor. 8.5]{Kyprianou2014:book:levy}, and the identities \eqref{idenB} in Lemma \ref{prop:levy} hold for all $u,v\geq0$.

For any $c,b\in\mathbb{R}$, defining the first passage times
$$\tau_{b}^{+}:=\inf\{t\geq 0: X_t>b\}
\quad\text{and}\quad
\tau_{c}^{-}:=\inf\{t\geq 0: X_t<c\} $$
with the convention that $\inf\emptyset=\infty$, we have the following result.

\begin{lem}\label{prop:levy}
For $c\leq x\leq b$ and $q,u,v\geq0$, we have
\begin{align}
\bE_{x}\big(e^{-q \tau_{b}^{+}}; \tau_{b}^{+}<\tau_{c}^{-}\big)
=&\ \frac{W^{(q)}(x-c)}{W^{(q)}(b-c)},\label{idenA}\\
\bE_{x}\big(e^{-u \tau_{c}^{-}+v X(\tau_{c}^{-})}; \tau_{c}^{-}<\tau_{b}^{+}\big)
=&\ e^{vx}\Big(Z^{(p)}_{v}(x-c)- \frac{W^{(p)}_{v}(x-c)}{W^{(p)}_{v}(b-c)} Z^{(p)}_{v}(b-c)\Big)\label{idenB}
\end{align}
where $p=u-\psi(v)$, and for $x\in(c,b)$
\begin{align}
\bE_{x}\big(e^{-q \tau_{c}^{-}}; X(\tau_{c}^{-})=c, \tau_{c}^{-}<\tau_{b}^{+}\big)
=&\ \frac{\sigma^{2}}{2}\Big(W^{(q)\prime}(x-c)- W^{(q)}(x-c)\frac{W^{(q)\prime}(b-c)}{W^{(q)}(b-c)}\Big).\label{idenC}
\end{align}
In addition, the resolvent of process $X$ killed at first exiting interval $[c,b]$ is specified by
\begin{align*}
&\ \int_{0}^{\infty} e^{-q t}
\bE_{x}\big(f(X_{t}), t<\tau^{+}_{b}\wedge\tau_{c}^{-}\big) dt\\
=&\ \int_{c}^{b} f(y)
\Big(\frac{W^{(q)}(x-c)}{W^{(q)}(b-c)}W^{(q)}(b-y)- W^{(q)}(x-y)\Big) dy.
\end{align*}
\end{lem}

Identity \eqref{idenA} can be found in \cite[Thm. 8.1]{Kyprianou2014:book:levy}.
To obtain the joint Laplace in \eqref{idenB}, we apply
 \cite[Thm.8.1]{Kyprianou2014:book:levy} under the new measure $\bP_{x}^{(v)}$.
Identity \eqref{idenC} can be found in  \cite[eqn (2.6)]{MIJATOVIC20123812},
and $\{X(\tau_{c}^{-})=c \}$ is  known as the creeping event,
which happens for a SNLP when the first downward passage over a level occurs by hitting the level with a positive probability. The result shows that a
SNLP creeps downwards if and only if it has a Gaussian component.

Let $\tau^{\{a\}}:=\inf\{t>0, X_{t}=a\}$ be the first hitting time. We could not find the following result in
literature and  provide a proof for the readers' convenience.

\begin{lem}\label{lem_hitting}
For $x,a\in(c,b)$, we have
\begin{equation}
\label{eqn:hitting:classical}
\bE_{x}\big(e^{-q \tau^{\{a\}}}; \tau^{\{a\}}<\tau_{b}^{+}\wedge\tau_{c}^{-}\big)
=\frac{W^{(q)}(x-c)}{W^{(q)}(a-c)}- \frac{W^{(q)}(x-a)}{W^{(q)}(b-a)} \frac{W^{(q)}(b-c)}{W^{(q)}(a-c)}.
\end{equation}
\end{lem}

\begin{proof}[Proof of Lemma \ref{lem_hitting}]
As observed in \cite[Lem 11]{Ivanovs2012:occupationdensity:map} that
$
\{\tau_{b}^{+}<\tau^{\{a\}}\}
=\{\tau_{b}^{+}<\tau_{a}^{-}\}
$
for $a<b$,
applying the strong Markov property of $X$ at $\tau^{\{a\}}$,
we have
\begin{align*}
&\ \bE_{x}\big(e^{-q\tau_{b}^{+}}; \tau_{b}^{+}<\tau_{c}^{-}\big)\\
=&\ \bE_{x}\big(e^{-q\tau_{b}^{+}}; \tau^{\{a\}}<\tau_{b}^{+}<\tau_{c}^{-}\big)
+\bE_{x}\big(e^{-q\tau_{b}^{+}}; \tau_{b}^{+}<\tau_{c}^{-}\wedge\tau^{\{a\}}\big)\\
=&\ \bE_{x}\big(e^{-q\tau^{\{a\}}}; \tau^{\{a\}}<\tau_{b}^{+}\wedge\tau_{c}^{-}\big)\cdot
\bE_{a}\big(e^{-q\tau_{b}^{+}}; \tau_{b}^{+}<\tau_{c}^{-}\big)
+ \bE_{x}\big(e^{-q \tau_{b}^{+}}; \tau_{b}^{+}<\tau_{a}^{-}\big).
\end{align*}
The  Laplace transform \eqref{eqn:hitting:classical} for the hitting time follows by applying \eqref{idenA}.
\end{proof}

\section{Main results }\label{main}

Write $\bar{X}_t:=\sup_{0\leq s\leq t}X_s $ for the running maximum
process for $X$.
The process of $X$ reflected at its running maximum is defined by
$Y_{t}:=\bar{X}_{t}-X_{t}$.
Let $\xi(\cdot)$ be a measurable function on $\mathbb{R}$.
Define the draw-down time for $X$ with respect to the draw-down function
$\xi$ as
$$
\tau_{\xi}:=\inf\{t> 0, X_t<\xi(\bar{X_t})\}
=\inf\{t>0, Y_{t}>\bar{\xi}(\bar{X}_{t})\}
$$
where $\bar{\xi}(z):=z-\xi(z)$ and $\{\xi(\bar{X}_{t}), t\geq 0 \}$ is the associated  draw-down level process.
For the case of constant function $\bar{\xi}$,
a sextuple law
was found in \cite{MIJATOVIC20123812}, where the concerned quantities include
the time of reaching the last maximum and the minimum value of $X$ before $\tau_{\xi}$ together with the undershoot at $\tau_{\xi}$.
The process $Y$ is referred to as the draw-down process in \cite{MIJATOVIC20123812}.
In this paper, we focus on an arbitrary measurable and strictly positive function $\bar{\xi}$ on $\mathbb{R}$
and we assume that $W^{(q)\prime}(x)$ exists on $(0,\infty)$ for simplicity.

In this section we first present expressions for
the solutions to the two-sided exit problems involving $\tau_{\xi}$.


\begin{prop}\label{prop:passagetime} For any $q>0$ and $x<b$, we have
\begin{equation}
\bE_{x}\big(e^{-q \tau_{b}^{+}}; \tau_{b}^{+}<\tau_{\xi}\big)
= \exp\Big(-\int_{x}^{b} \frac{W^{(q)\prime}(\bar{\xi}(y))}{W^{(q)}(\bar{\xi}(y))} dy\Big).
\label{eqn:b<xi}
\end{equation}
For any $u, v>0$, $k\in\mathbb{R}$ and $x<b$, with $p:=u-\psi(v)$ we have
\begin{equation}\label{eqn:xi<b}
\begin{aligned}
&\ \bE_{x} \big(e^{-u\tau_{\xi}+v X(\tau_{\xi})+k \bar{X}(\tau_{\xi})}; \tau_{\xi}<\tau_{b}^{+}\big)\\
=&\ e^{vx}
\int_{x}^{b}
e^{ky-\int_{x}^{y}
\frac{W_{v}^{(p)\prime}(\bar{\xi}(z))}{W_{v}^{(p)}(\bar{\xi}(z))} dz}
\Big(\frac{W^{(p)\prime}_{v}(\bar{\xi}(y))}{W^{(p)}_{v}(\bar{\xi}(y))} Z^{(p)}_{v}(\bar{\xi}(y))
- p W^{(p)}_{v}(\bar{\xi}(y))\Big) dy.
\end{aligned}
\end{equation}
\end{prop}

\begin{rem}
The assumption $\bar{\xi}>0$ on $[x,b]$ is necessary.
In fact, if $\bar{\xi}(a)=0$ for some $a\in(x,b)$,
one can find that $\tau_{\xi}\leq \tau_{a}^{+}$  $\bP_{x}$-a.s. by definition,
and \eqref{eqn:b<xi} and \eqref{eqn:xi<b} fail to hold.
It is often the case that the process $X$ is bounded from above by a constant $b$ for the event of interest. When this happens,   under $\bP_{x}$
the effective domain of $\bar{\xi}$ is $[x,b]$ instead of $\mathbb{R}$, i.e. only  the values of $\bar{\xi}(y)$ for $y\in [x, b]$ really matter, and the conditions on $\bar{\xi}(y)$  only need to be imposed for $y\in [x, b]$.
\end{rem}

For the potential measure up to time $\tau_{b}^{+}\wedge\tau_{\xi}$, we also have the next result.

\begin{prop}\label{prop:resolvent} For any $x<b$, we have
\begin{align}
&\ \int_{0}^{\infty} e^{-qt}
\bP_{x}\big(X_{t}\in dy; t<\tau_{b}^{+}\wedge\tau_{\xi}\big) dt
\non\\
=&\ \Big( \int_{x}^{b}
e^{-\int_{x}^{z} \frac{W^{(q)\prime}(\bar{\xi}(s))}{W^{(q)}(\bar{\xi}(s))} ds}
\Big(W^{(q)\prime}(z-y)- \frac{W^{(q)\prime}(\bar{\xi}(z))}{W^{(q)}(\bar{\xi}(z))} W^{(q)}(z-y)\Big) \mathbf{1}_{\{y\in(\xi(z), z)\}} dz\Big) dy\non\\
&\quad +W(0) \Big(
e^{-\int_{x}^{y} \frac{W^{(q)\prime}(\bar{\xi}(z))}{W^{(q)}(\bar{\xi}(z))} dz}
\mathbf{1}_{\{y\in(x,b)\}}\Big) dy.\label{eqn:resolvent}
\end{align}
\end{prop}

\begin{rem}\label{rem:5} The resolvent density in \eqref{eqn:resolvent} consists of two parts, where the second term degenerates if process $X$ has sample paths of unbounded variation.
By further analysis, one can find that
it is contributed by the total amount of time in
\[\mathcal{L}:=\{t>0: X_{t}=\bar{X}_{t}\}\]
until time $\tau_{b}^{+}\wedge\tau_{\xi}$, that is,
$$
\int_{0}^{\infty} e^{-q t} \mathbf{1}_{\{t\in\mathcal{L}\}} \mathbf{1}_{\{t<\tau_{b}^{+}\wedge\tau_{\xi}\}} \mathbf{1}_{\{X_{t}\in dy\}} dt;
$$
and for the case of $W(0)=0$, we have that $\bP_{x}$
almost surely, $\mathcal{L}$ is a Lebesgue null set.
\end{rem}

\begin{rem}\label{rem:2} If  $\xi\equiv c$ for some $c<x$, we have from Proposition \ref{prop:resolvent} that the support of the resolvent is $[c,b]$.
For any $y\in (c,b)$, since $y>\xi(z)=c$ in the following integral, we have
\begin{align*}
&\ \int_{x}^{b}
e^{-\int_{x}^{z} \frac{W^{(q)\prime}(\bar{\xi}(u))}{W^{(q)}(\bar{\xi}(u))} du}
\Big(W^{(q)\prime}(z-y)- \frac{W^{(q)\prime}(\bar{\xi}(z))}{W^{(q)}(\bar{\xi}(z))} W^{(q)}(z-y)\Big) \mathbf{1}_{\{y\in(\xi(z), z)\}} dz\\
=&\ \int_{x\vee y}^{b}
e^{-\int_{x}^{z} \frac{W^{(q)\prime}(u-c)}{W^{(q)}(u-c)} du}
\Big(W^{(q)\prime}(z-y)-
\frac{W^{(q)\prime}(z-c)}{W^{(q)}(z-c)} W^{(q)}(z-y)
\Big) dz\\
=& \Big(
e^{-\int_{x}^{z} \frac{W^{(q)\prime}(u-c)}{W^{(q)}(u-c)} du}
W^{(q)}(z-y)\Big)\Big|_{x\vee y}^{b}
=\ W^{(q)}(x-c) \Big(\frac{W^{(q)}(z-y)}{W^{(q)}(z-c)}\Big)\Big|^{b}_{x\vee y}\\
=&\ \frac{W^{(q)}(x-c)}{W^{(q)}(b-c)}W^{(q)}(b-y)
- W^{(q)}(x-y) \mathbf{1}_{\{x\geq y\}}
- W(0) \frac{W^{(q)}(x-c)}{W^{(q)}(y-c)} \mathbf{1}_{\{y>x\}}
\end{align*}
and
\begin{align*}
W(0) e^{-\int_{x}^{y} \frac{W^{(q)\prime}(\bar{\xi}(z))}{W^{(q)}(\bar{\xi}(z))} dz}\mathbf{1}_{\{y>x\}}
&\ = W(0) e^{-\int_{x}^{y} \frac{W^{(q)\prime}(z-c)}{W^{(q)}(z-c)} dz}\mathbf{1}_{\{y>x\}}\\
&\ = W(0) \frac{W^{(q)}(x-c)}{W^{(q)}(y-c)} \mathbf{1}_{\{y>x\}}.
\end{align*}
Making use of the fact that $W^{(q)}(z)=0$ for $z<0$ again, we recover the expression of the classical potential density.
\end{rem}

Notice that $\tau_{\xi}\wedge\tau_{c}^{-}=\tau_{\xi\vee c}$ for any $c<x$ and initial value $x$.
Applying Proposition \ref{prop:resolvent}, in the next result we can also obtain a joint distribution involving the running minimum and maximum before $\tau_{\xi}$ together with $X_{\tau_{\xi}-}$ and $X_{\tau_{\xi}}$ when there is an overshoot at the draw-down time $\tau_{\xi}$.

Let $\underline{X}_{t}:=\inf_{s\in[0,t]}X_{s}$ be the running minimum process of $X$.

\begin{cor}\label{cor:1}
For any nonnegative measurable function $f$ on $\mathbb{R}^{2}$ satisfying $f(z,z)=0$ for all $z\in\mathbb{R}$ and  for any $c<x<b$,
we have
\begin{align*}
&\ \bE_{x}\big(e^{-q \tau_{\xi}}
f(X_{\tau_{\xi}-}, X_{\tau_{\xi}});
\underline{X}_{\tau_{\xi}-}>c,
\bar{X}_{\tau_{\xi}}\leq b\big)\\
=&\ W(0) \int_{c}^{b}\Big(e^{-\int_{x}^{z} \frac{W^{(q)\prime}(\overline{\xi\vee c}(s))}{W^{(q)}(\overline{\xi\vee c}(s))} ds}\Big)\,dz \int^{-\bar{\xi}(z)}_{-\infty} f(z,z+u) \Pi(du)\\
&\quad+ \int_{c}^{b}\,dz  \int_{\xi(z)\vee c}^{z}\,dy \int^{\xi(z)-y}_{-\infty} f(y,y+u)\Pi(du)\\
&\quad\quad\times\Big(e^{-\int_{x}^{z} \frac{W^{(q)\prime}(\overline{\xi\vee c}(s))}{W^{(q)}(\overline{\xi\vee c}(s))} ds}
W^{(q)\prime}(z-y)- \frac{W^{(q)\prime}(\overline{\xi\vee c}(z))}{W^{(q)}(\overline{\xi\vee c}(z))} W^{(q)}(z-y)\Big).
\end{align*}
\end{cor}

We now consider the hitting problem of a draw-down level.
Denote by
$$\tau^{\{\xi\}}:=\inf\{t>0, X_{t}=\xi(\bar{X}_{t})\}
=\inf\{t>0, Y_{t}=\bar{\xi}(\bar{X}_{t})\}$$
the first time for $X$ to hit the draw-down level $\xi(\bar{X})$.
A particular interesting case of hitting is the event of creeping,
$\{\tau^{\{\xi\}}=\tau_{\xi}\}$,
which happens for a SNLP when the first downward passage over a level occurs by hitting the level with a positive probability.
It is well known that the classical creeping of a fixed level happens only if
$\sigma>0$, i.e. only if process $X$ has a nontrivial Brownian motion component.
If the downward passage time is replaced with a draw-down time, observe that within the duration of  each downward sample path of excursion away from the running maximum, the draw-down level $\xi({\bar{X}}) $ remains  constant. Therefore, one would expect that the draw-down creeping occurs if and only if process $X$ has a nontrivial Brownian motion component.

In the following proposition, another draw-down level with draw-down function $\theta(z)$ is introduced
with $\bar{\theta}(z)=z-\theta(z)>0$ for all $z\in\mathbb{R}$.

\begin{prop}\label{prop:hitting}
For any $x<b$ and $I:=\{z\in\mathbb{R}: \theta(z)<\xi(z)\}$, we have
\begin{align}
&\ \bE_{x}\Big(e^{-q \tau_{\xi}};
\tau^{\{\xi\}}=\tau_{\xi}<\tau_{b}^{+}\wedge \tau_{\theta}\Big)\non\\
=&\ \frac{\sigma^{2}}{2}
\int_{[x,b]\cap I}
e^{-\int_{x}^{y}\frac{W^{(q)\prime}(\overline{\xi\vee\theta}(z))}{W^{(q)}(\overline{\xi\vee\theta}(z))} dz}
\Big(\frac{(W^{(q)\prime}(\bar{\xi}(y)))^{2}}{W^{(q)}(\bar{\xi}(y))}- W^{(q)\prime\prime}(\bar{\xi}(y))\Big) dy.
\label{eqn:creeping}
\end{align}
Moreover, for any $x<b$ we have
\begin{align}
&\ \bE_{x}\Big(e^{-q \tau^{\{\xi\}}}; \tau^{\{\xi\}}<\tau_{b}^{+}\wedge \tau_{\theta}\Big)\non\\
=&\ \int_{[x,b]\cap I}
e^{-\int_{x}^{y}\frac{W^{(q)\prime}(\overline{\xi\vee\theta}(z))}
{W^{(q)}(\overline{\xi\vee\theta}(z))} dz}
\frac{W^{(q)}(\bar{\theta}(y))}{W^{(q)}(\bar{\theta}(y)-\bar{\xi}(y))}
\Big(\frac{W^{(q)\prime}(\bar{\xi}(y))}{W^{(q)}(\bar{\xi}(y))}-
\frac{W^{(q)\prime}(\bar{\theta}(y))}{W^{(q)}(\bar{\theta}(y))}\Big) dy.
\label{eqn:hitting}
\end{align}
\end{prop}

\begin{rem}
If $\theta<\xi=c$ for some $c<x<b$,
Proposition \ref{prop:hitting}  reduces to the classical result  \eqref{idenC}.
In this case, $I=\mathbb{R}$ and
$\overline{\xi\vee \theta}(y)=y-c=\bar{\xi}(y)$.
Similar to Remark \ref{rem:2}, we have
\[
\int_{x}^{y}\frac{W^{(q)\prime}(\overline{\xi\vee\theta}(z))}{W^{(q)}(\overline{\xi\vee\theta}(z))} dz
=\int_{x}^{y}\frac{W^{(q)\prime}(z-c)}{W^{(q)}(z-c)}\,dz=
\log\Big(\frac{W^{(q)}(y-c)}{W^{(q)}(x-c)}\Big)
\quad\text{for $y\in(x,b)$}.
\]
Then the right hand side of \eqref{eqn:creeping} equals to
\begin{align*}
&\ \frac{\sigma^{2}}{2}\int_{x}^{b} \frac{W^{(q)}(x-c)}{W^{(q)}(y-c)}
\Big(\frac{(W^{(q)\prime}(y-c))^{2}}{W^{(q)}(y-c)}- W^{(q)\prime\prime}(y-c)\Big)\,dy\\
=&\ \frac{\sigma^{2}}{2}
W^{(q)}(x-c)\cdot \Big(-\frac{W^{(q)\prime}(y-c)}{W^{(q)}(y-c)}\Big)\Big|_{x}^{b}\\
=&\ \frac{\sigma^{2}}{2}
\Big(W^{(q)\prime}(x-c)- W^{(q)}(x-c)\frac{W^{(q)\prime}(b-c)}{W^{(q)}(b-c)}\Big),
\end{align*}
which recovers \eqref{idenC}.

Similarly, one can recover Lemma \ref{lem_hitting} from \eqref{eqn:hitting} by taking
$\xi=a, \theta=c$ with $c<a<b$ and $x\in(c,b)$.
\end{rem}

It is also interesting to study similar problems associated to the draw-up times (from the running minimum) with a general draw-up function for a spectrally negative L\'evy process. But it seems challenging to express the desired results in terms of the scale functions.

\section{Applications}
In this section we present two applications of the results from Section \ref{main}.

\subsection*{Selling a stock at a draw-down time}
\begin{ex}\label{exam:1}
The decision to sell a stock is a combination of art and science.
In general, it is ideal  to sell a stock at a price as high as possible
or right before it starts to decline.
However, very few investors can buy at the absolute bottom and sell at the absolute high.
If one does not sell at the right timing, the profit disappears.
There are a number of considerations to decide when  the best time is.
In this example,
we assume that the price process of an underlying security
is given by $S=\{S_{t}=e^{X_{t}},t\geq0\}$.
The investor sells a stock either when it hits a price target in order to lock in gains
or before the ratio $S/\overline{S}$ leaves too far below $1$ to stop the loss, where $\overline{S}$ is the historical high process for  $S$.
Using the Cobb-Douglas function, the investor sells out of a stock when the utility process
\[\{\D \big({S_{t}}/{S_{0}}\big)^{\gamma}\big({S_{t}}/{\overline{S}_{t}}\big)^{1-\gamma}, \,\, t\geq 0  \}  \]
leaves a pre-determined interval $[a,b]$ for some $\gamma,a\in(0,1)$ and $b>1$.
It can be checked directly that $\D T_{a,b}=T^{+}_{a,b}\wedge T^{-}_{a,b}$ where
\[
T^{+}_{a,b}:=\inf\{t>0: S_{t}>b^{\frac{1}{\gamma}} S_{0}\}
\quad\text{and}\quad
T^{-}_{a,b}:=\inf\{t>0: S_{t}/S_{0}< a\cdot(\overline{S}_{t}/S_{0})^{1-\gamma} \}.
\]
Without loss of generality, we take $S_{0}=1$ and then $X_{0}=0$.
Formulated in terms of the first passage time for $X$ under $\bP$, we have
$\D T^{+}_{a,b}=\tau^{+}_{\frac{\log b}{\gamma}}$
and
$\D T^{-}_{a,b}=\tau_{\xi}$ for $\xi(z)=(1-\gamma) z+\log a$.
Let $q>0$ be the risk-free interest rate and $p=q-\psi(1)$. Then we have
 \begin{equation}\label{eqn:exam:1:1}
 \bE_{1}\Big(e^{-q T_{a,b}} S_{T_{a,b}}; T_{a,b}=T^{+}_{a,b}\Big)
 =b^{\frac{1}{\gamma}}\Big(\frac{W^{(q)}(\log 1/a )}{W^{(q)}(\log b/a)}\Big)^{\frac{1}{\gamma}}
 \end{equation}
 and
 \begin{equation}\label{eqn:exam:1:2}
 \begin{aligned}
 \bE_{1}\Big(e^{-q T_{a,b}} S_{T_{a,b}};&\ T_{a,b}=T^{-}_{a,b}\Big)
=Z^{(p)}_{1}(\log 1/a)
- Z^{(p)}_{1}(\log b/a)
\big(\frac{W^{(p)}_{1}(\log 1/a)}{W^{(p)}_{1}(\log b/a)}\big)^{\frac{1}{\gamma}}\\
&\ -p \frac{1-\gamma}{\gamma}
\int_{\log 1/a}^{\log b/a}
\big(W^{(p)}_{1}(\log 1/a)\big)^{\frac{1}{\gamma}}
\big(W^{(p)}_{1}(y)\big)^{\frac{\gamma-1}{\gamma}}\,dy.
 \end{aligned}
 \end{equation}
 \end{ex}
\begin{proof}[Proof of \eqref{eqn:exam:1:1} and \eqref{eqn:exam:1:2}]
In this case, for $z\in[0,\frac{\log b}{\gamma}]$,
$\bar{\xi}(z)=\gamma z- \log a>0$ by definition.
Then $\bar{\xi}'(z)=\gamma$ and we  have for $x\in[0,\frac{\log b}{\gamma}]$,
\begin{equation}\label{eqn:exam:1:3}
 \int_{0}^{x}\frac{W^{(q)\prime}(\bar{\xi}(y))}{W^{(q)}(\bar{\xi}(y))}\,dy
= \frac{1}{\gamma}\Big(\log\big(W^{(q)}(\bar{\xi}(y))\big)\Big)\Big|^{x}_{0}
= \frac{1}{\gamma}\log\Big(\frac{W^{(q)}(\bar{\xi}(x))}{W^{(q)}(\bar{\xi}(0))}\Big).
\end{equation}
Applying the formula \eqref{eqn:b<xi} and
the fact that $S_{T^{+}_{a,b}}=b^{\frac{1}{\gamma}}$ on the set $\{T^{+}_{a,b}<\infty\}$
gives \eqref{eqn:exam:1:1}.

Similarly, we have
\begin{align*}
&\ \int_{0}^{\frac{\log b}{\gamma}}
\exp\Big(-\int_{0}^{y}
\frac{W^{(p)\prime}_{1}(\bar{\xi}(z))}{W^{(p)}_{1}(\bar{\xi}(z))}dz\Big)
\Big(
\frac{W^{(p)\prime}_{1}(\bar{\xi}(y))}{W^{(p)}_{1}(\bar{\xi}(y))}
Z^{(p)}_{1}(\bar{\xi}(y))- p \gamma W^{(p)}_{1}(\bar{\xi}(y))\Big)\,dy\\
=&\
\Big(-Z^{(p)}_{1}(\bar{\xi}(y))
e^{-\int_{0}^{y}
\frac{W^{(p)\prime}_{1}(\bar{\xi}(z))}{W^{(p)}_{1}(\bar{\xi}(z))}dz
}\Big)\Big|_{0}^{\frac{\log b}{\gamma}}
= Z^{(p)}_{1}(\log \frac{1}{a})
- Z^{(p)}_{1}(\log \frac{b}{a})
\big(\frac{W^{(p)}_{1}(\log \frac{1}{a})}{W^{(p)}_{1}(\log \frac{b}{a})}\big)^{\frac{1}{\gamma}}.
\end{align*}
The identity \eqref{eqn:exam:1:2} can be proved by applying \eqref{eqn:xi<b} and a change of variable argument.
\end{proof}

\subsection*{First passage times of the reflected SNLP}
\begin{ex}\label{exam:3}
For function $\xi(z)=z-d$ for some $d>0$, the draw-down time $\tau_{\xi}$ reduces to the first passage time of SNLP reflected at its running maximum which was investigated in \cite{Avram2004:levyexit:annals, Pistorius2004:passagetime:reflect},
that is, $\tau_{\xi}=\kappa^{+}_{d}$ for
$$
\kappa^{+}_{d}:=\inf\{t>0, Y_{t}>d\}.
$$
We have, for $b, q, u, v>0$ and $k\in\mathbb{R}$,
\begin{align}
&\ \bE\big(e^{-q \tau_{b}^{+}}; \tau_{b}^{+}<\kappa^{+}_{d}\big)
= \exp\Big(- \frac{W^{(q)\prime}(d)}{W^{(q)}(d)} b\Big),
\label{eqn:cor2:1}
\end{align}

\begin{align}
&\ \bE\big(e^{-u\kappa^{+}_{d}+v X(\kappa^{+}_{d})+k\bar{X}(\kappa^{+}_{d})}; \kappa^{+}_{d}<\tau_{b}^{+}\big)\non\\
=&\
\Big(\frac{W^{(p)\prime}_{v}(d)}{W^{(p)}_{v}(d)} Z^{(p)}_{v}(d)
- p W^{(p)}_{v}(d)\Big)
\frac{1- e^{(k-\frac{W^{(p)\prime}_{v}(d)}{W^{(p)}_{v}(d)}) b}}
{\frac{W^{(p)\prime}_{v}(d)}{W^{(p)}_{v}(d)}-k},
\label{eqn:cor2:2}
\end{align}
and
\begin{align}
&\ \bE\big(e^{-q \kappa^{+}_{d}}; Y(\kappa^{+}_{d})=d, \kappa^{+}_{d}<\tau_{b}^{+}\big)
\non\\
=&\ \frac{\sigma^{2}}{2}
 \Big(W^{(q)\prime}(d) -
 \frac{W^{(q)}(d)W^{(q)\prime\prime}(d)}{W^{(q)\prime}(d)}\Big)
\big(1- e^{- \frac{W^{(q)\prime}(d)}{W^{(q)}(d)} b}\big),
\label{eqn:cor2:3}
\end{align}
where $p=u-\psi(v)$. In addition, we have for $y\in(-d,b)$
\begin{align}\label{eqn:cor2:4}
&\ \int_{0}^{\infty} e^{-qt}\bP\Big(X_{t}\in dy, t<\kappa^{+}_{d}\wedge\tau_{b}^{+}\Big) dt
\nonumber\\
=&\ \Big(W^{(q)}(d\wedge(b-y)) e^{- \frac{W^{(q)\prime}(d)}{W^{(q)}(d)}(y+d)\wedge b}
- W^{(q)}(-y)\Big) dy.
\end{align}

\end{ex}


Note that  $\bar{\xi}(t)\equiv d$ in this case,
the above results
follow directly from Propositions \ref{prop:passagetime}, \ref{prop:hitting}
and \ref{prop:resolvent}, respectively,
where an argument similar to Remark \ref{rem:2} is applied in obtaining \eqref{eqn:cor2:4}.


\begin{rem}
Denoting by $e_{q}$ an exponential random variable with parameter $q$ and independent of $X$,
it follows from \eqref{eqn:cor2:1} that $\bar{X}_{e_{q}\wedge\kappa^{+}_{d}}$ is exponentially distributed with parameter $\frac{W^{(q)\prime}(d)}{W^{(q)}(d)}$ since
$$\bE\big(e^{-q \tau_{b}^{+}}; \tau_{b}^{+}<\kappa^{+}_{d}\big)
=\bP\big(\tau_{b}^{+}<e_{q}\wedge\kappa^{+}_{d}\big)
=\bP\big(\bar{X}_{e_{q}\wedge\kappa^{+}_{d}}>b\big).$$
Taking $k=-v<0$ in \eqref{eqn:cor2:2} and letting $b\to\infty$, we have
$$
\bE\big(e^{-u\kappa^{+}_{d}- v Y(\kappa^{+}_{d})}; \kappa^{+}_{d}<\infty\big)
=Z^{(p)}_{v}(d)
-W^{(p)}_{v}(d)
\frac{pW^{(p)}_{v}(d)+ v Z^{(p)}_{v}(d)}
{W^{(p)\prime}_{v}(d)+ vW^{(p)}_{v}(d)},
$$
which coincides with \cite[Thm.1]{Avram2004:levyexit:annals}.
\end{rem}

\begin{rem}
For $0<b-y<d$ and $b>0$, identity \eqref{eqn:cor2:4} can  be rewritten as
$$
 \bP\big(X_{e_{q}}\in dy, e_{q}<\kappa^{+}_{d}, \bar{X}_{e_{q}}<b\big)
= q \Big(W^{(q)}(b-y) e^{- \frac{W^{(q)\prime}(d)}{W^{(q)}(d)}b}
- W^{(q)}(-y)\Big) dy.$$
Then
\begin{align*}
& \bP\big(X_{e_{q}}\in dy, e_{q}<\kappa^{+}_{d}, \bar{X}_{e_{q}}\in db\big)\\
=&\ q \Big(W^{(q)\prime}(b-y)- \frac{W^{(q)\prime}(d)}{W^{(q)}(d)} W^{(q)}(b-y)\Big)
e^{- \frac{W^{(q)\prime}(d)}{W^{(q)}(d)}b} dy db.
\end{align*}
Therefore, for $z\in(0,d)$ we have
\begin{align*}
 \bP\Big(Y_{e_{q}}\in dz, e_{q}<\kappa^{+}_{d}\Big)
=&\ q \Big(W^{(q)\prime}(z)- \frac{W^{(q)\prime}(d)}{W^{(q)}(d)}W^{(q)}(z)\Big)
\int_{0}^{\infty} e^{-\frac{W^{(q)\prime}(d)}{W^{(q)}(d)}b} db\\
=&\ q \Big(W^{(q)}(d)\frac{W^{(q)\prime}(z)}{W^{(q)\prime}(d)}- W^{(q)}(z)\Big) dz,
\end{align*}
which coincides with the resolvent given in \cite[Thm.1.(ii)]{Pistorius2004:passagetime:reflect}.

By Remark \ref{rem:5}, for the case $W(0)>0$, we have for $y\in(0,b)$
$$
 \bP(e_{q}<\kappa^{+}_{d}\wedge \tau_{b}^{+}, \bar{X}(e_{q})\in dy,
Y(e_{q})=0)
= q W(0) e^{-\frac{W^{(q)\prime}(d)}{W^{(q)}(d)} y} dy.
$$
Therefore,
$$
\bP(e_{q}<\kappa^{+}_{d}\wedge \tau_{b}^{+}, Y(e_{q})=0)
=\ q \frac{W(0)W^{(q)}(d)}{W^{(q)\prime}(d)}\Big(1- e^{-\frac{W^{(q)\prime}(d)}{W^{(q)}(d)} b}\Big).$$
It follows that
$$\bP(Y(e_{q})=0, e_{q}<\kappa^{+}_{d})
=\ q W(0) \frac{W^{(q)}(d)}{W^{(q)\prime}(d)},
$$
which gives the time $Y$ spent at $0$ before $\kappa^{+}_{d}$ and coincides with \cite[Thm.1.(ii)]{Pistorius2004:passagetime:reflect}.
\end{rem}

\section{Proofs of main results}

This section is dedicated to the proofs for our main results,
where we make use of  the excursion theory for  Markov processes,
and appeal to the compensation formula and the exponential formula for Poisson point processes;
see for example \cite[O.5]{Bertoin96:book}.
To this end, we first restate the formula concerned in terms of excursions, and then apply the compensation formula.
For this we use the following notations from
\cite[IV]{Bertoin96:book}, \cite{Avram2004:levyexit:annals} and \cite{Pis07},
and refer the readers to the book for a detailed discussion on the related excursion theory.
Moreover, by the spatial homogeneity of $X$, we mainly focus on the cases under $\bP$. More
general results for $\bP_{x}$ can be derived by a shifting argument
as shown in the proof of Proposition \ref{prop:passagetime}.

Recall that $\bar{X}$ is the running maximum process of $X$ and let $Y:=\bar{X}-X$ be the reflected process.
It is known that $Y$ is a 'nice' Markov process with
$0$ being instantaneous whenever $W(0)=0$.
Let $\mathscr{L}:=\{t>0, Y_{t}=0\}$ be the zero set of $Y$
and $\overline{\mathscr{L}}$ be its closure.
A local time process $L$ of $Y$ at $0$ is a continuous process
that increases only on $\overline{\mathscr{L}}$ and is unique up to a multiplicative factor.
Thus, there exists $\nu\geq0$ such that
\begin{equation}\label{eqn:2}
\int_{0}^{t} \mathbf{1}_{\{s\in\mathscr{L}\}}\,ds= \nu L(t)
\quad \text{for all } t\geq0;
\end{equation}
see \cite[Cor.IV.6]{Bertoin96:book}.
The right inverse of $L$ is defined by
\[L^{-1}_{t}:=\inf\{s>0: L(s)>t\}, \,\,\, t\geq 0.\]
Under the new time scale,
the excursion process of $Y$ away from zero,
associated to $L$ and denoted by $\epsilon\equiv\{\epsilon_{r}, r\geq0\}$,
takes values in the so-called excursion space of paths away from $0$
with an additional isolated point $\gamma$, $\mathscr{E}\cup\{\gamma\}$,
and is defined by
$$
\epsilon_{r}:=\{Y_{t}, L^{-1}_{r-}\leq t< L^{-1}_{r}\}
\quad\text{if\ } L^{-1}_{r-}<L^{-1}_{r},
$$
and $\epsilon_{r}:=\gamma$ otherwise. The excursion process
$\epsilon$ is a Poisson point process,
possibly stopped at time $L(\infty)$ with an excursion of infinite lifetime,
 characterized by a $\sigma$-finite measure $n(\cdot)$ on $\mathscr{E}$. Set
$\big(\overline{\mathscr{L}}\big)^{c}$ consists of countable excursion intervals,
and $\mathscr{L}$ differs from $\overline{\mathscr{L}}$ by at  most countable points.
In particular,
\[\D \int_{0}^{L^{-1}_{r}} \mathbf{1}_{\{s\in\overline{\mathscr{L}}\}}\,ds= \nu r\]
on $\{L^{-1}_{r}<\infty\}$ under $\bP$; see \cite[Lem.VI.8]{Bertoin96:book}.
For a Borel function $f\geq0$ on $\mathscr{E}\cup\{\gamma\}$ with $f(\gamma)=0$,
we write
\[\D n(f):=\int_{\mathscr{E}} f(\varepsilon) n(d\varepsilon).\]
For any $c\geq0$, $n^{(c)}$ denotes the  associated excursion measure for $X$ under $\bP^{(c)}$.

For a SNLP $X$ under $\bP$, its running maximum process $\bar{X}$ fulfills the condition of a local time
and is  chosen to be the local time of $Y$ at $0$; see \cite[VII]{Bertoin96:book}.
For this choice of $L$,
\[\D L(t)=\sup_{s\in[0,t]}X_{s} \]
and $L^{-1}_{s}=\tau_{s}^{+}$ is a subordinator with Laplace exponent $\Phi$.
Since $\nu$ is the drift parameter of $L^{-1}$ (see \cite[Thm.IV.8]{Bertoin96:book}),
we have
\[\D \nu=\lim_{s\to\infty}\frac{\Phi(s)}{s}=\lim_{s\to\infty}\frac{s}{\psi(s)}=W(0).\]
For an excursion $\varepsilon\in \mathscr{E}$,
its lifetime is denoted by $\zeta$
and its excursion height is denoted by $\overline{\varepsilon}$.
The first passage and the first hitting time of $\varepsilon$ are respectively defined by
$$
\rho_{c}^{+}:=\inf\{s\in(0,\zeta): \varepsilon(s)> c\}
\quad \text {and} \quad
\rho^{\{c\}}:=\inf\{s\in(0,\zeta): \varepsilon(s)=c\},
$$
with the convention that $\inf \emptyset =\infty$.
We write $\rho^{+}_{c}(r)$, $\rho^{\{c\}}(r)$ and $\zeta(r)$
for the first passage time, the first hitting time and the life time, respectively,
of the excursion at local time $r$, that is, $\epsilon_{r}=\varepsilon\in\mathscr{E}$.
As before we denote by $e_{q}$ an exponential variable with parameter $q>0$ and independent of $X$.

\begin{proof}[Proof of Proposition \ref{prop:passagetime}]

 Observe from their definitions that under $\bP$, 
\begin{itemize}
\item
on the set $\{\tau_{b}^{+}<\infty\}$,
$$
\tau_{b}^{+}
=L^{-1}_{b}
=\int_{0}^{L^{-1}_{b}} \mathbf{1}_{\{t\in\overline{\mathscr{L}}\}}\,dt
+\int_{0}^{L^{-1}_{b}} \mathbf{1}_{\{t\notin\overline{\mathscr{L}}\}}\,dt
=\nu b+ \sum_{r\in[0,b]} \zeta(r);
$$
\item
on the set $\{\tau_{\xi}<\infty\}$,
$\tau_{\xi}=L^{-1}_{r-}+\rho^{+}_{\bar{\xi}(r)}(r)$
and
$\bar{X}(\tau_{\xi})=r$
 for $r=L(\tau_{\xi})$;
\item
on the set $\{\tau^{\{\xi\}}<\infty\}$,
$\tau^{\{\xi\}}=L^{-1}_{r-}+\rho^{\{\bar{\xi}(r)\}}(r)$
and
$\bar{X}(\tau_{\xi})=r$ for $r=L(\tau^{\{\xi\}})$.
\end{itemize}

From the idea in the proof of \cite[Thm.VII.8]{Bertoin96:book},
it holds that
$$
\{\tau_{b}^{+}<\tau_{\xi}\}
=
\{\bar{\epsilon}_{r}\leq \bar{\xi}(r)\ \text{for all\ } r\in[0,b]\}.
$$
Therefore, we have from the exponential compensation formula,
see \cite[O.5]{Bertoin96:book}, that
\begin{align}
&\ \bE \big(e^{-q \tau_{b}^{+}}; \tau_{b}^{+}<\tau_{\xi}\big)
= e^{-q \nu b}\times\bE \Big(
e^{-q \sum_{r\in[0,b]}\zeta(r)} \prod_{r\in[0,b]}
\mathbf{1}_{\{\bar{\epsilon}_{r}\leq \bar{\xi}(r)\}}\Big)\non\\
=&\ e^{-q \nu b}\times \bE\Big(\exp\big(- \sum_{r\in[0,b]}
\big(\zeta(r)+ \infty\cdot \mathbf{1}_{\{\bar{\epsilon}_{r}> \bar{\xi}(r)\}}\big)\big)\Big) \non\\
=&\ \exp\Big(-\int_{0}^{b} \Big(q \nu+
\int_{\mathscr{E}}\big(1- e^{-q \zeta} \mathbf{1}_{\{\overline{\varepsilon}\leq\bar{\xi}(r)\}}\big)
n(d\varepsilon)\Big)\,dr\Big),\label{eqn:1}
\end{align}
with the understanding that $e^{-\infty}=0$ and $\infty\times 0=0$.
To solve the problem,
we recall the classical case of constant $\xi\equiv c<0$ in Lemma \ref{prop:levy}.
\begin{align*}
&\
\frac{W^{(q)}(-c)}{W^{(q)}(b-c)}=\bE\big(e^{-q\tau_{b}^{+}}; \tau_{b}^{+}<\tau_{c}^{-}\big)\\
=&\ \exp\Big(-\int_{0}^{b} \Big(q \nu+
\int_{\mathscr{E}}\big(1- e^{-q \zeta} \mathbf{1}_{\{\overline{\varepsilon}\leq (r-c)\}}\big)
n(d\varepsilon)\Big)\,dr\Big).
\end{align*}
Differentiating in $b$ on the both sides of the equation above gives
\begin{equation}
q \nu+
\int_{\mathscr{E}}\big(1- e^{-q \zeta} \mathbf{1}_{\{\overline{\varepsilon}\leq z\}}\big)
 n(d\varepsilon)
= \frac{W^{(q)\prime}(z)}{W^{(q)}(z)}
\quad \text{for $z>0$.}
\label{eqn:lem:exc:1}
\end{equation}
Applying the formula \eqref{eqn:lem:exc:1} to the previous equation \eqref{eqn:1}
gives \eqref{eqn:b<xi} under $\bP$.

For the event $\{\tau_{\xi}<\tau_{b}^{+}\}$ and for $p=u-\psi(v)>0$, taking a change of measure we have
\begin{align*}
&\ \bE\big(e^{-u\tau_{\xi}+v X(\tau_{\xi})+k\bar{X}(\tau_{\xi})}; \,\tau_{\xi}<\tau_{b}^{+}\big)\\
&= \bE^{(v)} \big(e^{-p\tau_{\xi}+k\bar{X}(\tau_{\xi})}; \,\tau_{\xi}<\tau_{b}^{+}\big)\\
&=\ \bE^{(v)}\Big(\sum_{r\in[0,b]}
\Big(e^{kr-p L^{-1}_{r-}}
\prod_{s<r} \mathbf{1}_{\{\overline{\epsilon}_{s}\leq \bar{\xi}(s)\}}\Big)
\times e^{-p\rho^{+}_{\bar{\xi}(r)}(r)}
\mathbf{1}_{\{\bar{\epsilon}_{r}>\bar{\xi}(r)\}}\Big)\\
&=\ \int_{0}^{b}e^{kr}\bE^{(v)} \Big(e^{-p L^{-1}_{r-}} \prod_{s<r}
\mathbf{1}_{\{\overline{\epsilon}_{s}\leq \bar{\xi}(s)\}}\Big)\times
n^{(v)}\big(e^{-p\rho^{+}_{\bar{\xi}(r)}}; \overline{\varepsilon}>\bar{\xi}(r)\big)\,dr\\
&=\ \int_{0}^{b}e^{kr} \bE^{(v)} \Big(e^{-p L^{-1}_{r}}; L^{-1}_{r}<\tau_{\xi}\Big)
\times
n^{(v)}\big(e^{-p \rho^{+}_{\bar{\xi}(r)}}; \overline{\varepsilon}>\bar{\xi}(r)\big)\,dr,
\end{align*}
where a compensation formula is applied for the third equality; see e.g. \cite[O.5]{Bertoin96:book},
and where we used the fact that $L^{-1}_{t}\neq L^{-1}_{t-}$ for at most countably many values of $t$.

Considering now the case where $\xi\equiv c<0$, $u=q>0$ and $v=k=0$, we have
\begin{align*}
&\
 Z^{(q)}(-c)- \frac{W^{(q)}(-c)}{W^{(q)}(b-c)} Z^{(q)}(b-c)\\
&=\ \bE \big(e^{-q\tau_{c}^{-}}; \tau_{c}^{-}<\tau_{b}^{+}\big)\\
&=\ \int_{0} ^{b} \bE \Big(e^{-q \tau_{r}^{+}}; \tau_{r}^{+}<\tau_{c}^{-}\Big)
\times n\big(e^{- q \rho^{+}_{r-c}}; \overline{\varepsilon}>(r-c)\big)\,dr\\
&=\ \int_{0} ^{b} \frac{W^{(q)}(-c)}{W^{(q)}(r-c)}\times
n\big(e^{- q \rho^{+}_{r-c}}; \overline{\varepsilon}>(r-c)\big)\,dr.
\end{align*}
Differentiating in $b$ on the both sides of the above equation gives
\begin{equation}
n\big(e^{-q \rho^{+}_{z}}; \overline{\varepsilon}>z\big)
= \frac{W^{(q)\prime}(z)}{W^{(q)}(z)} Z^{(q)}(z)- q W^{(q)}(z)
\quad\text{for $z>0$.}
\label{eqn:lem:exc:2}
\end{equation}
Plugging \eqref{eqn:b<xi} and \eqref{eqn:lem:exc:2} into the equation gives
the formula \eqref{eqn:xi<b} under $\bP$.
The general result for $u,v>0$ and $k\in\mathbb{R}$ follows by an analytic extension.

 To consider the general case of $X(0)=x$ with $x<b$,
 we introduce a function $\D \varsigma(y):=\xi(y+x)-x$. Then
 \[ \bar{\varsigma}(y)=y+x-\xi(y+x)=\bar{\xi}(y+x).\]
 Since $X$ is spatially homogenous, we have
 \[(X,\bar{X},\tau_{\xi})|_{\bP_{x}}=(x+X, x+\bar{X}, \tau_{\varsigma})|_{\bP}.\]
  Therefore,
\begin{align*}
&\ \bE_{x}\big(e^{-q \tau_{b}^{+}}; \tau_{b}^{+}<\tau_{\xi}\big)
=\bE\big(e^{-q \tau_{b-x}^{+}}; \tau_{b-x}^{+}<\tau_{\varsigma}\big)\\
=&\ \exp\Big(-\int_{0}^{b-x} \frac{W^{(q)\prime}(\bar{\varsigma}(y))}
{W^{(q)}(\bar{\varsigma}(y))} dy\Big)
= \exp\Big(-\int_{x}^{b}
\frac{W^{(q)\prime}(\bar{\xi}(y))}
{W^{(q)}(\bar{\xi}(y))}\,dy\Big).
\end{align*}
which gives \eqref{eqn:b<xi}. Similarly, we have from the spatial homogeneity of $X$ that
\begin{align*}
&\ \bE_{x} \big(e^{-u\tau_{\xi}+v X(\tau_{\xi})+k\bar{X}(\tau_{\xi})}; \tau_{\xi}<\tau_{b}^{+}\big)\\
&= e^{vx+kx}\bE\big(e^{-u\tau_{\varsigma}+v X(\tau_{\varsigma})+k\bar{X}(\tau_{\varsigma})};
\tau_{\varsigma}<\tau_{b-x}^{+}\big) \\
&=\ e^{vx+kx} \int_{0}^{b-x}
e^{ky-\int_{0}^{y}\frac{W_{v}^{(p)\prime}(\bar{\varsigma}(z))}{W_{v}^{(p)}(\bar{\varsigma}(z))} dz}
\Big(\frac{W^{(p)\prime}_{v}(\bar{\varsigma}(y))}{W^{(p)}_{v}(\bar{\varsigma}(y))} Z^{(p)}_{v}(\bar{\varsigma}(y))- p W^{(p)}_{v}(\bar{\varsigma}(y))\Big) dy\\
&=\ e^{vx}\int_{x}^{b}
e^{ky-\int_{x}^{y} \frac{W_{v}^{(p)\prime}(\bar{\xi}(z))}{W_{v}^{(p)}(\bar{\xi}(z))} dz}
\Big(\frac{W^{(p)\prime}_{v}(\bar{\xi}(y))}{W^{(p)}_{v}(\bar{\xi}(y))} Z^{(p)}_{v}(\bar{\xi}(y))
- p W^{(p)}_{v}(\bar{\xi}(y))\Big) dy.
\end{align*}
This concludes the proof of Proposition \ref{prop:passagetime}.
\end{proof}

In the following proofs, we use
an idea from \cite[Lem.VI.8]{Bertoin96:book} 
and the compensation formula,
and we only focus on the case under $\bP$ and for  $b>0$.

\begin{proof}[Proof of Proposition \ref{prop:resolvent}]
Let $f\geq0$ be a bounded and continuous function on $\mathbb{R}$.
For the resolvent of $X$ killed at $\tau_{b}^{+}\wedge\tau_{\xi}$,
which is defined in \eqref{eqn:resolvent}, we have
\begin{align*}
&\ \int_{0}^{\infty} e^{-qt}\bE \big(f(X_{t}); t<\tau_{b}^{+}\wedge\tau_{\xi}\big)\,dt\\
=&\ \int_{0}^{\infty} e^{-qt}\bE \big(f(X_{t}), t\in\overline{\mathscr{L}}, t<\tau_{b}^{+}\wedge\tau_{\xi}\big)\,dt\\
&\quad + \int_{0}^{\infty} e^{-qt}\bE \big(f(X_{t}), t\notin\overline{\mathscr{L}}, t<\tau_{b}^{+}\wedge\tau_{\xi}\big)\,dt\\
=:&I_{1}+I_{2}.
\end{align*}
Recalling equation \eqref{eqn:2},
applying Fubini's Theorem and a change of variable, we have
\begin{align*}
I_{1}=&\ \nu\bE \Big(\int_{0}^{\infty} e^{-q t} f(X_{t}) \mathbf{1}_{\{t<\tau_{b}^{+}\wedge\tau_{\xi}\}}\,dL_{t}\Big)\\
=&\ \nu \bE \Big(\int_{0}^{\infty} e^{-q L^{-1}_{r}}
f(X(L^{-1}_{r})) \mathbf{1}_{\{L^{-1}_{r}<\tau_{b}^{+}\wedge\tau_{\xi}\}}\,dr\Big)\\
=&\ \nu \int_{0}^{b} f(r)
\bE \big( e^{-q L^{-1}_{r}}; L^{-1}_{r}<\tau_{\xi}\big)\,dr,
\end{align*}
where note that $L^{-1}_{r}=\tau_{r}^{+}$ and $X(L^{-1}_{r})=r$ on event $\{L^{-1}_{r}<\infty\}$.
On the other hand, we have
\begin{align*}
&\ q I_{2}= \bE \Big(f(X_{e_{q}}); e_{q}\notin \overline{\mathscr{L}}, e_{q}<\tau_{b}^{+}\wedge\tau_{\xi}\Big)\\
=&\ \bE \Big(\sum_{r\in[0,b]}
f\big(r-\epsilon_{r}(e_{q}-L^{-1}_{r-})\big)\times
\prod_{s<r} \mathbf{1}_{\{\overline{\epsilon}_{s}\leq \bar{\xi}(s)\}}
\mathbf{1}_{\{e_{q}<L^{-1}_{r-}+\rho^{+}_{\bar{\xi}(r)}(r)\}};
L^{-1}_{r-}<e_{q}<L^{-1}_{r}\Big),
\end{align*}
where $\bar{X}_{e_{q}}=L(e_{q})=r$ on the event $\{L^{-1}_{r-}<e_{q}<L^{-1}_{r}\}$.
By the memoryless property of $e_{q}$ and the compensation formula,
we further have
\begin{align*}
&\ q I_{2}=
\bE \Big(
\sum_{r\in[0,b]}e^{-qL^{-1}_{r-}}
\prod_{s<r} \mathbf{1}_{\{\overline{\epsilon}_{s}\leq \bar{\xi}(s)\}}
\times f(r-\epsilon_{r}(e_{q})) \mathbf{1}_{\{e_{q}<\rho^{+}_{\bar{\xi}(r)}(r)\wedge\zeta(r)\}}\Big)\\
&\quad= \int_{0}^{b} \bE \big(e^{-q L^{-1}_{r}}; L^{-1}_{r}<\tau_{\xi}\big)\times
q \int_{0}^{\infty}n\Big(e^{-qs}f(r-\varepsilon(s));
s<\rho^{+}_{\bar{\xi}(r)}\wedge\zeta\Big)\,dsdr.
\end{align*}
Putting these together gives
\begin{align*}
&\ \int_{0}^{\infty} e^{-qt}\bE \big(f(X_{t}); t<\tau_{b}^{+}\wedge\tau_{\xi}\big)\,dt\\
=&\
\int_{0}^{b}
\bE \big(e^{-q \tau_{r}^{+}}; \tau_{r}^{+}<\tau_{\xi}\big)\times
\Big( \nu f(r)+ \int_{0}^{\infty}
e^{-qs} n\big(f(r-\varepsilon(s));
s<\rho^{+}_{\bar{\xi}(r)}\wedge\zeta\big)\,ds\Big)dr.
\end{align*}

For the  case of $\xi\equiv c<0$,
we have
\begin{align*}
&\ \int_{0}^{\infty} e^{-qt}\bE \big(f(X_{t}); t<\tau_{b}^{+}\wedge\tau_{c}^{-}\big)\,dt\\
=&\ \int_{0}^{b} \frac{W^{(q)}(-c)}{W^{(q)}(r-c)}
\Big(\nu f(r)+ \int_{0}^{\infty} e^{-q s} n\big( f(r-\varepsilon(s));
s<\rho^{+}_{r-c}\wedge\zeta\big)\,ds\Big)\,dr\\
=&\ \int_{c}^{b} f(y) \Big(\frac{W^{(q)}(-c)}{W^{(q)}(b-c)} W^{(q)}(b-y)-W^{(q)}(-y)\Big)\,dy.
\end{align*}
Further differentiating in $b$ on the above equation,
we eventually have
\begin{align}
&\ \nu f(b)+ \int_{0}^{\infty} e^{-q s} n\big(f(b-\varepsilon(s)); s<\rho^{+}_{b-c}\wedge\zeta\big)\,ds\non\\
=&\ W(0)f(b)+ \int_{c}^{b}f(y)
\Big(W^{(q)\prime}(b-y)- \frac{W^{(q)\prime}(b-c)}{W^{(q)}(b-c)}W^{(q)}(b-y)\Big)\,dy.
\label{eqn:lem:exc:5}
\end{align}
Since $\nu=W(0)$,  formula \eqref{eqn:resolvent} is thus proved for $x=0$.
\end{proof}

\begin{proof}[Proof of Proposition \ref{prop:hitting}]
Similar to the proofs of Propositions \ref{prop:passagetime} and \ref{prop:resolvent},
for the event of creeping $\{\tau^{\{\xi\}}=\tau_{\xi}<\tau_{b}^{+}\wedge \tau_{\theta}\}$,
i.e. the event that the draw-down event happens before $X$
leaves interval $[\theta(\bar{X}),b]$
by hitting the draw-down level $\xi(\bar{X}) $, we have
\begin{align*}
&\ \bE \big(e^{-q \tau_{\xi}}; \tau^{\{\xi\}}=\tau_{\xi}<\tau_{b}^{+}\wedge \tau_{\theta}\big)\\
=& \bE \Big(\sum_{r\in[0,b]} \big(e^{-q L^{-1}_{r-}}
\prod_{s<r} \mathbf{1}_{\{\overline{\epsilon}_{s}\leq \bar{\theta}(s)\}}
\mathbf{1}_{\{\overline{\epsilon}_{s}\leq \bar{\xi}(s)\}}\big)
\times \big(e^{-q \rho^{+}_{\bar{\xi}(r)}(r)}
\mathbf{1}_{\{\epsilon_{r}(\rho^{+}_{\bar{\xi}(r)})=\bar{\xi}(r)\}}
\mathbf{1}_{\{\bar{\xi}(r)<\bar{\theta}(r)\}}
\big)\Big)\\
=&\ \int_{0}^{b} \bE \big(e^{- q L^{-1}_{r}}; L^{-1}_{r}<\tau_{\xi\vee \theta}\big)
\times n\big(e^{-q \rho^{+}_{\bar{\xi}(r)}}; \varepsilon(\rho^{+}_{\bar{\xi}(r)})=\bar{\xi}(r) \big)
\times \mathbf{1}_{\{r\in I\}}
\,dr.
\end{align*}
For the case of $\xi(z)=\theta(z)+1\equiv c$, we have
$I=\mathbb{R}$ and
$\{\tau_{\xi}\leq \tau_{b}^{+}\wedge \tau_{\theta}\}=\{\tau_{c}^{-}\leq \tau_{b}^{+}\}$
\begin{align*}
&\ \bE \big(e^{-q \tau_{c}^{-}}; X(\tau_{c}^{-})=c, \tau_{c}^{-}<\tau_{b}^{+}\big)\\
&= \frac{\sigma^{2}}{2} \Big(W^{(q)\prime}(-c)- W^{(q)}(-c) \frac{W^{(q)\prime}(b-c)}{W^{(q)}(b-c)}\Big)\\
&= \int_{0}^{b} \frac{W^{(q)}(-c)}{W^{(q)}(r-c)} \times
n\big(e^{-q \rho^{+}_{r-c}}; \varepsilon(\rho^{+}_{r-c})=r-c\big)\,dr.
\end{align*}
Differentiating on the both sides of the equation above, we have
\begin{equation}
n\big(e^{-q \rho^{+}_{z}}; \varepsilon(\rho^{+}_{z})=z\big)
= \frac{\sigma^{2}}{2}\Big(\frac{(W^{(q)\prime}(z))^{2}}{W^{(q)}(z)}- W^{(q)\prime\prime}(z)\Big)\quad\text{for $z>0$.}
\label{eqn:lem:exc:3}
\end{equation}
Identity \eqref{eqn:creeping} is thus proved by applying \eqref{eqn:b<xi} for $x=0$.

The hitting of a maximum dependent level
cannot be derived by applying the strong Markov property of $X$
as in the classical case in Lemma \ref{lem_hitting}.
However, due to the absence of positive jumps,
a similar observation is that
$$
\{L^{-1}_{s}<\tau^{\{\xi\}}\wedge \tau_{\theta}\}
=\{L^{-1}_{s}<\tau_{\xi}\wedge \tau_{\theta}\}
\quad\text{on the event $\{L^{-1}_{s}<\infty\}$,}
$$
that is, every excursion at time $s<r=L(\tau^{\{\xi\}})$
fails to go above level $\overline{\xi\vee\theta}(s)$. Therefore,
\begin{align*}
&\ \bE \big(e^{-q \tau^{\{\xi\}}}; \tau^{\{\xi\}}<\tau_{b}^{+}\wedge\tau_{\theta}\big)\\
=&\ \bE \Big(\sum_{r\in[0,b]}
\Big(e^{-q L^{-1}_{r-}} \prod_{s<r}
\mathbf{1}_{\{\overline{\epsilon}_{s}\leq \bar{\xi}(s)\}}
\mathbf{1}_{\{\overline{\epsilon}_{s}\leq \bar{\theta}(s)\}}\Big)
\times \Big(e^{-q \rho^{\{\bar{\xi}(r)\}}(r)}
\mathbf{1}_{\{\rho^{\{\bar{\xi}(r)\}}(r)<\rho^{+}_{\bar{\theta}(r)}(r)\}}\Big)\Big)\\
=&\ \int_{0}^{b} \bE \big(e^{-q L^{-1}_{r}}; L^{-1}_{r}\leq \tau_{\xi\vee\theta}\big)
\times n\Big(e^{-q \rho^{\{\bar{\xi}(r)\}}}; \rho^{\{\bar{\xi}(r)\}}<\rho^{+}_{\bar{\theta}(r)}\Big)
\mathbf{1}_{\{\bar{\xi}(r)<\bar{\theta}(r)\}}\,dr.
\end{align*}
For the case of $\xi\equiv a$ and $\theta\equiv c$ with $c<a<0<b$, we have
\begin{align*}
&\ \frac{W^{(q)}(-c)}{W^{(q)}(a-c)}- \frac{W^{(q)}(-a)W^{(q)}(b-c)}{W^{(q)}(b-a)W^{(q)}(a-c)}\\
=&\bE \Big(e^{-q\tau^{\{a\}}}; \tau^{\{a\}}<\tau_{b}^{+}\wedge\tau_{c}^{-}\Big)\\
=&\ \int_{0}^{b} \bE \big(e^{-q \tau_{r}^{+}}; \tau_{r}^{+}<\tau_{a}^{-}\big)
\times n \Big(e^{- q \rho^{\{r-a\}}}; \rho^{\{r-a\}}<\rho^{+}_{r-c}\Big)\,dr.
\end{align*}
Thus, differentiating in $b$ gives for $b>a>c$
\begin{equation}
n\Big(e^{-q \rho^{\{b-a\}}}; \rho^{\{b-a\}}<\rho^{+}_{b-c}\Big)
=\frac{W^{(q)}(b-c)}{W^{(q)}(a-c)}
\Big(\frac{W^{(q)\prime}(b-a)}{W^{(q)}(b-a)}- \frac{W^{(q)\prime}(b-c)}{W^{(q)}(b-c)}\Big)
\label{eqn:lem:exc:4}
\end{equation}
which proves identity \eqref{eqn:hitting}.
\end{proof}

\begin{rem}
We remark that the
excursion theory for a (reflected) SNLP has been employed to solve various problems.
The formulas
\eqref{eqn:lem:exc:2}, \eqref{eqn:lem:exc:5}, \eqref{eqn:lem:exc:3},
\eqref{eqn:lem:exc:1} for $q=0$ and \eqref{eqn:lem:exc:4} for $b-c=\infty$
have also been found as well; see \cite{Pis07} for a collection of such results.
\end{rem}

\begin{proof}[Proof of Corollary \ref{cor:1}]
Since
$\{t<\tau_{b}^{+}\}=\{\bar{X}_{t}< b\}$
and $\{t<\tau_{c}^{-}\}=\{\underline{X}_{t}> c\}$
for every $t>0$,
we have from Proposition \ref{prop:resolvent} that, for $z>x\geq c$ and $\xi(z)\vee c<y\leq z$,
\begin{align*}
&\ \int_{0}^{\infty} e^{-q t} \bP_{x}(X_{t}\in dy, \bar{X}_{t}\in dz , t<\tau_{\xi}, \underline{X}_{t}\geq c)\,dt\\
=&\ \int_{0}^{\infty} e^{-q t} \bP_{x}(X_{t}\in dy, \bar{X}_{t}\in dz , t<\tau_{\xi\vee c})\,dt\\
=&\
\Big(e^{-\int_{x}^{z} \frac{W^{(q)\prime}(\overline{\xi\vee c}(s))}{W^{(q)}(\overline{\xi\vee c}(s))} ds}
W^{(q)\prime}(z-y)- \frac{W^{(q)\prime}(\overline{\xi\vee c}(z))}{W^{(q)}(\overline{\xi\vee c}(z))} W^{(q)}(z-y)\Big)
dzdy\non\\
&\quad + W(0) \Big(
e^{-\int_{x}^{z} \frac{W^{(q)\prime}(\overline{\xi\vee c}(s))}{W^{(q)}(\overline{\xi\vee c}(s))} ds}\Big) dz\cdot \delta_{z}(dy),
\end{align*}
where $\delta_{z}$ denotes a Direct measure at $z$.
Notice that $\tau_{\xi}\in \{t>0, X_{t-}\neq X_{t}\}$ on the event
$\{X_{\tau_{\xi}}\neq X_{\tau_{\xi}-}\}\cap\{\tau_{\xi}<\infty\}$, and $\{(X_{t}-X_{t-}), t\geq0\}$ can be identified as
a Poisson point process with characteristic measure $\Pi(\cdot)$. Therefore, for $f\geq0$ satisfying $f(z,z)=0$, we have from the Fubini's theorem that
\begin{align*}
&\ \bE_{x}\big(e^{-q \tau_{\xi}} f(X_{\tau_{\xi}-}, X_{\tau_{\xi}}); \underline{X}_{\tau_{\xi}-}\geq c, \bar{X}_{\tau_{\xi}}\leq b\big)\\
=&\ \bE_{x}\big(e^{-q \tau_{\xi}} f(X_{\tau_{\xi}-}, X_{\tau_{\xi}}); X_{\tau_{\xi}}\neq X_{\tau_{\xi}-},
\underline{X}_{\tau_{\xi}-}\geq c, \bar{X}_{\tau_{\xi}}\leq b\big)\\
=&\ \bE_{x}\Big(\sum_{\{t: X_{t-}\neq X_{t}\}} e^{-q t} f(X_{t-}, X_{t-}+\Delta X_{t})
\mathbf{1}(\underline{X}_{t-}\geq c)\mathbf{1}(\bar{X}_{t}\leq b)\\
&\quad\quad\quad \times\mathbf{1}(t\leq\tau_{\xi})
\mathbf{1}(X_{t-}+\Delta X_{t}<\xi(\bar{X}_{t}))\Big)\\
=&\ \int_{0}^{\infty}e^{-q t}\,dt \int_{x}^{b} \int_{c}^{z}  \bP_{x}(X_{t-}\in dy, \bar{X}_{t-}\in dz , t<\tau_{\xi}, \underline{X}_{t-}\geq c)\\
\qquad\quad&\qquad\times \int_{-\infty}^{\xi(z)-y}f(y,y+u)\Pi(du),
\end{align*}
where the compensation formula is applied for the last equation.
The desired result then follows from the quasi-left continuity for process $X$.
\end{proof}





\acks
The authors are thankful to anonymous referees for careful reading and helpful comments and suggestions.
Bo Li is supported by National Natural Science Foundation of China (No. 11601243).
Bo Li, Nhat Linh Vu and Xiaowen Zhou are supported by NSERC (RGPIN-2016-06704).
Xiaowen Zhou is supported by Natural Science Foundation of Hunan Province (No. 2017JJ2274)


%
%
%
%


\end{document}